\documentclass[11pt]{amsart}

\title{Extension operators and geometric decompositions}
\author{Yakov Berchenko-Kogan}

\usepackage{amsthm,braket,hyperref,tikz}
\usepackage[backend=bibtex, maxbibnames=5]{biblatex}
\bibliography{references.bib}
\usetikzlibrary{cd}

\newtheorem{theorem}{Theorem}[section]
\newtheorem{proposition}[theorem]{Proposition}
\newtheorem{corollary}[theorem]{Corollary}
\newtheorem{lemma}[theorem]{Lemma}
\theoremstyle{definition}
\newtheorem{definition}[theorem]{Definition}

\newtheorem{example}[theorem]{Example}
\newtheorem{remark}[theorem]{Remark}

\newcommand\fT{\mathfrak T}
\newcommand\fS{\mathfrak S}

\newcommand\oF{\mathring{\cF}}
\newcommand\cT{\mathcal T}
\newcommand\cF{\mathcal F}
\newcommand\cI{\mathcal I}
\newcommand\cC{\mathcal I^\perp}
\newcommand\R{\mathbb R}
\newcommand\cP{\mathcal P}
\newcommand\oP{\mathring{\cP}}

\renewcommand\phi\varphi

\DeclareMathOperator\starop\star
\DeclareMathOperator\tr{tr}

\newcommand\D{\mathop{}\!d}
\newcommand\dl{\D\lambda}

\begin{document}

\begin{abstract}
  Geometric decomposition is a widely used tool for constructing local bases for finite element spaces. For finite element spaces of differential forms on simplicial meshes, Arnold, Falk, and Winther showed that geometric decompositions can be constructed from extension operators satisfying certain properties. In this paper, we generalize their results to function spaces and meshes satisfying very minimal hypotheses, while at the same time reducing the conditions that must hold for the extension operators. In particular, the geometry of the mesh and the mesh elements can be completely arbitrary, and the function spaces need only have well-defined restrictions to subelements. In this general context, we show that extension operators yield geometric decompositions for both the function space and its dual. Later, we specialize to simplicial meshes, and we show that, to obtain geometric decompositions, one needs only to construct extension operators on the reference simplex in each dimension. In particular, for simplicial meshes, the existence of geometric decompositions depends only on the dimension of the mesh.
\end{abstract}

\maketitle

\section{Introduction}
Local bases for finite element spaces yield sparse matrices when implementing the finite element method, resulting in efficient computation. In \cite{afw09}, Arnold, Falk, and Winther showed that local bases arise from \emph{geometric decompositions}, that is, decompositions of the global function space as a direct sum of function spaces associated to mesh elements. In turn, Arnold, Falk, and Winther showed that one can construct a geometric decomposition from what they call a \emph{consistent family of extension operators}. Arnold, Falk, and Winther's framework of geometric decomposition has proven to be useful in a wide variety of finite element contexts; see, for example, \cite{bk24, chchhuwe24, chra16, giklsa19}. However, the results in \cite{afw09} are limited to finite element spaces of differential forms on simplicial meshes, and constructing extension operators and checking consistency is not always easy. Consequently, in new contexts, geometric decompositions have been constructed by way of analogy with \cite{afw09}, rather than by direct application of results. To address this issue, this paper extends the results in \cite{afw09} to very general meshes and function spaces, while at the same time simplifying the conditions needed for the results to apply.

Specifically, our results can apply to meshes regardless of the geometry (simplicial, hexahedral, mixed) of the elements. In fact, the only things we need to know about the mesh are what the elements are, and which elements are subelements of which. In other words, we need a partially ordered set, which we denote $\fT$. So, for example, if $\cT$ is a triangulation, then $\fT$ will be the set of vertices, edges, faces, etc., of the triangulation, and we say that $K\le F$ if $K$ is contained in $F$. Our function spaces are likewise very general: we need to be able to associate a vector space $\cF(F)$ to every element $F\in\fT$, and we need to have restriction maps $\tr^F_K\colon\cF(F)\to\cF(K)$ whenever $K$ is a subelement of $F$. So, for example, $\cF$ could be the $\cP_r\Lambda^k$ or $\cP_r^-\Lambda^k$ spaces of Arnold, Falk, and Winther \cite{afw06}, and the trace maps would then be pullbacks via the inclusion maps $K\hookrightarrow F$.

We note here that geometric decompositions do not always exist. A simple example is the space $\mathcal P_0(\cT)$ of continuous piecewise constant scalar fields. Presuming $\cT$ is connected, such scalar fields must of course be globally constant, so this space is just the one-dimensional space of constant scalar fields on $\cT$. In particular, $\mathcal P_0(\cT)$ does not have a local basis. We give a less trivial example at the end of the paper in Example~\ref{eg:constant}.

\subsection{Motivation}\label{subsec:motivation}

If we apply our results to the standard finite element exterior calculus (FEEC) context of differential forms \cite{araw14, arbobo15, afw06}, we obtain the same geometric decompositions as in \cite{afw09}, presuming that we start with the same function spaces and extension operators. Although this paper provides some insight into the choices that go into the primal and dual geometric decompositions of FEEC, the main motivation for this work is to go beyond FEEC and obtain geometric decompositions in new contexts of interest, all within a single framework. Overall, motivated by our study of finite element spaces for double forms \cite{bega25double} and blow-up finite elements \cite{bega25blow}, this paper is focused on distilling simple essential conditions needed to obtain geometric decompositions in the sense of \cite{afw09}, in the most general context possible.

We discuss these motivating examples in greater detail. Instead of differential forms, we can more generally take the $\cF(F)$ to be spaces of covariant tensors, with the trace maps $\tr^F_K$ defined to be pullbacks via the inclusion $K\hookrightarrow F$ as before. Such spaces of covariant tensors include the Regge finite elements \cite{li2018regge}, and, more generally, double forms, that is, form-valued forms \cite{bega25double,huli25}, which have applications to elasticity, fluid mechanics, and numerical geometry. Indeed, in the author's join work with Evan Gawlik \cite{bega25double}, we apply Corollary~\ref{cor:simpext} of this paper as a key step in constructing finite element spaces of double forms. Specifically, the main result, \cite[Theorem~4.29]{bega25double}, should be viewed as constructing a simplicial extension operator in the sense of Definition~\ref{def:simpext}. From there, Corollary~\ref{cor:simpext} yields the finite element spaces in \cite[Section 5.1]{bega25double}. We expect a similar approach will work for covariant tensors more generally.

While double forms illustrate the benefit of more general function spaces, blow-up finite elements \cite{bega25blow} illustrate the benefit of working with general partially ordered sets instead of cell complexes (meshes). In \cite{bega25blow}, we construct finite element spaces suitable for discretizing tangent vector fields with full continuity on surface meshes, with expected applications to surface flows and numerical geometry. One notable feature of these finite element spaces is that the shape functions and degrees of freedom are associated not to faces of the triangulation but to increasing sequences of faces, known as \emph{flags}. For example, on a single triangle $T$ with vertices $V_0$, $V_1$, and $V_2$, we have the shape function $\psi=\lambda_0\frac{\lambda_1}{\lambda_1+\lambda_2}$ in barycentric coordinates. This shape function is multivalued at $V_0$, so we cannot consider traces to $V_0$ in the usual sense. Instead, \cite[Definition~2.24]{bega25blow} considers trace operators that are defined as sequences of limits associated to flags. So, for example, the trace to the flag $V_0<E_{01}<T$ is defined by sending $\lambda_2$ to zero first and then sending $\lambda_1$ to zero, so the trace of $\psi$ to $V_0<E_{01}<T$ is $1$. The trace of $\psi$ to the flag $E_{01}<T$ is defined in the usual way by setting $\lambda_2$ to zero, yielding $\lambda_0$. However, the trace to $V_0<T$ is novel: per \cite{bega25blow}, we send $\lambda_1$ and $\lambda_2$ both to zero, while maintaining their ratio, so the trace of $\psi$ to $V_0<T$ is $\frac{\lambda_1}{\lambda_1+\lambda_2}$; this represents the limit of $\psi$ as we approach $V_0$ along the ray determined by $\lambda_1$ and $\lambda_2$. We can also compose these, for example taking the trace to $V_0<T$ and then to $V_0<E_{01}<T$ by sending $\lambda_2$ to zero in the expression $\frac{\lambda_1}{\lambda_1+\lambda_2}$, obtaining $1$ as before. The framework in this paper allows for this more complex structure of the trace maps. In this context, rather than being vertices, edges, and faces ordered by inclusion, the elements of $\fT$ are flags ordered by refinement, so, for example $(V_0<E_{01}<T)<(V_0<T)$. See also the examples in \cite[Sections~3.3 and 3.4]{bega25blow}.

\subsection{Comparison with finite element systems}
There is some overlap between our framework and the framework of finite element systems (FES) \cite{ch08,chhu23,chhu18,chra16}; we will highlight the comparisons throughout the paper. One key difference is that FES works with cell complexes, whereas we work with partially ordered sets. Consequently, as discussed above, we are able to naturally accommodate degrees of freedom in \cite{bega25blow} that are associated to flags rather than the faces themselves. Furthermore, the fact that $\fT$ is only required to have the structure of a partially ordered set means that, in principle, rather than representing a mesh, it could more broadly represent the hierarchical structures that we see in mesh refinement, macroelements, etc. We will see a hint of this broader notion in this paper, where we view the mesh $\cT$ itself as an element of the partially ordered set, yielding the global function space $\cF(\cT)$ within the same framework as the single-element function spaces $\cF(F)$.

Another difference is that, while FES also obtains geometric decompositions for both the function space and its dual, our decomposition for the function space is local in the sense that functions associated to $K\in\fT$ are nonzero only on elements that contain $K$. This property holds, for example, for the Bernstein basis or, more generally, for the geometric decompositions in \cite{afw09}. See also Remark~\ref{rem:fes}.

As in FES \cite{chhu18}, the framework in this paper accommodates a more general perspective on the trace maps rather than just thinking about them as restrictions. Specializing to simplicial triangulations in two dimensions for clarity, we can think of $\cF$ on triangles as defining the local function spaces, $\cF$ on edges as defining the interelement continuity conditions, and $\cF$ on vertices as defining the relations between the interelement continuity conditions. Thus, as in \cite{chhu18}, we could use this framework to consider, for example, $C^1$ elements, where we can think of $\tr^T_E\colon\cF(T)\to\cF(E)$ as recording the values and normal derivatives of a function on the triangle $T$ at the edge $E$, so then $\cF(E)$ is a vector space of pairs of functions, recording the values and normal derivatives, respectively. There is also no requirement that the function spaces be finite-dimensional, although we do need sufficient regularity to define trace maps.

\subsection{Extension operators}

In addition to generalizing the meshes and function spaces for which we construct geometric decompositions, we also show that one can get away with a weaker notion of consistent extension operators than that in \cite{afw09}, with fewer things to construct and to check. Specifically, in \cite{afw09}, the authors construct geometric decompositions given extension operators $E_K^F\colon\cF(K)\to\cF(F)$ satisfying certain consistency properties. In this paper, as in \cite[Proposition~2.2]{chra16}, we show that it suffices to only have extension operators defined on the subspace of vanishing trace, that is, $E_K^F\colon\oF(K)\to\cF(F)$, where $\oF(K)$ is the subspace of $\cF(K)$ consisting of all $\phi$ that vanish when restricted to any proper subface of $K$. Consequently, in addition to having a more limited task in constructing the extension operators, the properties to check the extension operators' consistency become simpler. From these more limited extension operators, we still obtain the geometric decomposition. Moreover, if needed, we can come full circle and use the geometric decomposition to construct extension operators on the full space $\cF(K)\to\cF(F)$ that satisfy the consistency properties in \cite{afw09}, but we do not need these full operators a priori. See Remarks~\ref{rem:afw} and \ref{rem:fes} for more discussion about how the extension operators in this paper compare to those in \cite{afw09} and \cite{chra16}, respectively.

\subsection{Organization}
The bulk of the paper is Section~\ref{sec:decomp}, where we define our general notion of function spaces and extension operators, construct the geometric decomposition (Theorem~\ref{thm:decomp} and corollaries), and construct the dual decomposition (Theorem~\ref{thm:dualdecomp} and corollaries). In Section~\ref{sec:simplicial}, we apply the results of Section~\ref{sec:decomp} to simplicial triangulations. In this context, we show in Theorem~\ref{thm:simpext} that extension operators can be defined independently of the mesh: It suffices to construct extension operators $\oF(T^n)\to\cF(Q^{n+1})$, where $T^n$ is the simplex in the first orthant defined by $\lambda_0+\dots+\lambda_n=1$, and $Q^{n+1}$ is the simplex in the first orthant defined by $\lambda_0+\dots+\lambda_n\le1$. In other words, if one constructs extension operators $\oF(T^n)\to\cF(Q^{n+1})$ for a particular function space, then one obtains geometric decompositions for all simplicial meshes.

\section{Local extension operators and geometric decomposition}\label{sec:decomp}
\subsection{Function spaces and extension operators}\label{subsec:functionspaces}
We begin by defining function spaces on a partially ordered set. As discussed, although the context is quite general, an example to keep in mind is a triangulation $\cT$, in which case the partially ordered set $\fT$ is the set of vertices, edges, faces, etc., of the triangulation, ordered by inclusion. Meanwhile, an example of a function space $\cF$ is $\cP_r\Lambda^k$ or even simply the Lagrange space $\cP_r$.

Recall that a partially ordered set is a set equipped with a relation $\le$ satisfying the reflexivity property that $F\le F$, the antisymmetry property that $F\le G$ and $G\le F$ implies $F=G$, and the transitivity property that $K\le G$ and $G\le F$ implies $K\le F$.

\begin{definition}\label{def:funspace}
  Let $\fT$ be a partially ordered set. A \emph{function space} $\cF$ is
  \begin{itemize}
  \item a vector space $\cF(F)$ associated to every $F\in\fT$, and
  \item a linear map $\tr_K^F\colon\cF(F)\to\cF(K)$, called the \emph{trace map}, associated to every relation $K\le F$,
  \end{itemize}
  satisfying the functoriality properties that
  \begin{itemize}
  \item $\tr_F^F\colon\cF(F)\to\cF(F)$ is the identity, and
  \item $\tr_K^G\circ\tr_G^F=\tr_K^F$ whenever $K\le G\le F$.
  \end{itemize}
\end{definition}

\begin{definition}\label{def:vanishingtrace}
  We define the subspace of functions with \emph{vanishing trace} as
  \begin{equation*}
    \oF(F):=\Set{\phi\in\cF(F)|\tr_K^F\phi=0\text{ for all }K<F}.
  \end{equation*}
\end{definition}

We can construct global function spaces from local ones.

\begin{definition}\label{def:globalfunctionspace}
  Let $\cF$ be a function space on a partially ordered set $\fT$. We define the \emph{global function space} $\cF(\cT)$ via the inverse limit construction, that is
  \begin{equation*}
    \cF(\cT):=\Set{\left(\phi_F\right)_{F\in\fT}\in\prod_{F\in\fT}\cF(F)|\tr^F_K\phi_F=\phi_K\text{ for all }K\le F}.
  \end{equation*}
  We define $\tr^\cT_F\colon\cF(\cT)\to\cF(F)$ to be the natural projection maps.
\end{definition}

The above definitions match those of finite element systems \cite{chhu18}, except that we do not have the structure of differentials and evaluation maps, and we work on a partially ordered set instead of a cell complex. The following proposition gives an example of the utility of the partially ordered set perspective. Note that, in this treatment, $\cT$ is merely a symbol designed to match standard notation like $\cP_r(\cT)$; the structure is encoded by the order relation of $\fT$ and the trace operators of $\cF$.

\begin{proposition}\label{prop:functionhat}
  We can formally add $\cT$ to $\fT$ as the top element, yielding a partially ordered set $\hat\fT:=\fT\sqcup\{\cT\}$ with $F<\cT$ for all $F\in\fT$. Then the vector space $\cF(\cT)$ and the trace maps $\tr^\cT_F$ make $\cF$ a function space on $\hat\fT$, and we have $\oF(\cT)=0$.
\end{proposition}

\begin{proof}
  The only functoriality properties that need to be checked are that $\tr^F_K\tr^\cT_F=\tr^\cT_K$ whenever $K\le F$, which follows from the condition $\tr^F_K\phi_F=\phi_K$ in the definition of $\cF(\cT)$. Meanwhile, if $\phi=(\phi_F)_{F\in\fT}\in\oF(\cT)$, then each $\phi_F=\tr^\cT_F\phi$ is zero by definition, so $\phi=0$.
\end{proof}

Unlike in \cite{afw09}, our extension operators need only be defined on the space of vanishing trace.

\begin{definition}\label{def:local}
  Let $\cF$ be a function space on a partially ordered set $\fT$, and let $\cF(\cT)$ be the corresponding global function space. For $K\in\fT$, a \emph{local extension operator} is a linear map $E^{\cT}_K\colon\oF(K)\to\cF(\cT)$ such that
  \begin{itemize}
  \item $\tr^{\cT}_KE^{\cT}_K\colon\oF(K)\to\cF(K)$ is the inclusion, and
  \item $\tr^{\cT}_FE^{\cT}_K=0$ if $K\not\le F$.
  \end{itemize}
\end{definition}
One way of thinking about the second condition is that if $\psi$ is a nonzero element of $\oF(K)$ and $\phi$ is its extension $\phi=E^{\cT}_K\psi\in\cF(\cT)$, then $\phi$ cannot vanish on faces $F$ that contain $K$, for then it would vanish on $K$ as well. The second condition requires that $\phi$ vanish everywhere else. In particular, $\phi$ is local, since it is nonzero only on faces that contain $K$.

Equivalently, we can define extension operators without reference to the global function space.
\begin{definition}\label{def:consistent}
  Let $\cF$ be a function space on a partially ordered set $\fT$. We define a \emph{consistent family of extension operators} to be linear maps $E^F_K\colon\oF(K)\to\cF(F)$ for every $K,F\in\fT$, such that
  \begin{itemize}
  \item $E^K_K\colon\oF(K)\to\cF(K)$ is the inclusion,
  \item $E^F_K=0$ if $K\not\le F$, and
  \item $\tr^F_GE^F_K=E^G_K$ whenever $G\le F$.
  \end{itemize}
\end{definition}

Note that the extension maps that appear in the three properties are all of the form $E_K^\bullet$, so, as with local extension operators, the concept of a consistent family of extension operators makes sense even if we fix a particular $K$.

\begin{remark}\label{rem:afw}
  Our definition of a consistent family of extension operators differs from that in \cite{afw09}. In particular, we only require that the extension operators be defined on the spaces of vanishing trace, and our required properties are essentially \cite[Equation~(4.5) and Lemma~4.1]{afw09}. As we will show in Corollary~\ref{cor:afw} via the geometric decomposition, we can extend the definition of extension operators from the subspace of vanishing trace $\oF(K)$ to the full space $\cF(K)$, and these operators satisfy the definition of a consistent family of extension operators in \cite{afw09}. In other words, we take as our definition the simple ``immediate implications'' of Arnold, Falk, and Winther's definition, and we will eventually show that, via the geometric decomposition, these ``immediate implications'' actually imply Arnold, Falk, and Winther's original definition. In particular, if we apply our results to the extension operators given in \cite{afw09}, then we will obtain the same geometric decompositions for $\cP_r\Lambda^k$ and $\cP_r^-\Lambda^k$ as in \cite{afw09}.
\end{remark}

\begin{remark}\label{rem:fes}
  The idea that we only need to extend functions of vanishing trace also appears in \cite[Proposition~2.2]{chra16}. However, we require that extensions of functions in $\oF(K)$ vanish on all faces that do not contain $K$ (matching \cite[Lemma~4.1]{afw09}), whereas \cite{chra16} only requires vanishing on faces other than $K$ of the same dimension. So, for example, on a triangle $T$ with vertices $V_0$, $V_1$, and $V_2$, the function $\lambda_0^2+17\lambda_1\lambda_2\in\cP_2(T)$ is a valid extension of $1\in\cP_2(V_0)$ following \cite{chra16}, but it is not a valid extension following Definition~\ref{def:consistent}. Indeed, $\lambda_0^2+17\lambda_1\lambda_2$ vanishes on the other vertices $V_1$ and $V_2$, but it does not vanish on the edge $E_{12}$ even though $E_{12}$ does not contain $V_0$. To follow Definition~\ref{def:consistent}, we could instead use the extension $\lambda_0^2\in\cP_2(T)$, which vanishes on all of $E_{12}$. Consequently, the extensions in \cite[Proposition~2.2]{chra16} do not necessarily satisfy \cite[Lemma~4.1]{afw09} and hence do not necessarily yield local bases in the sense of \cite{afw09}. The advantage of using $\lambda_0^2$ instead of $\lambda_0^2+17\lambda_1\lambda_2$ is clear if we consider a triangulation. The function $\lambda_0^2$ is nonzero only on those elements that contain $V_0$, whereas $\lambda_0^2+17\lambda_1\lambda_2$ is additionally nonzero on the element across the edge $E_{12}$ from $T$. In this sense, $\lambda_0^2$ is ``more local'' than $\lambda_0^2+17\lambda_1\lambda_2$.
\end{remark}

We now prove an equivalence between local extension operators and consistent families of extension operators.

\begin{proposition}\label{prop:localconsistent}
  A local extension operator for every $K\in\fT$ yields a consistent family of extension operators via $E^F_K:=\tr^\cT_F E^\cT_K$. Conversely, a consistent family of extension operators yields local extension operators via $E^\cT_K:=\left(E_K^F\right)_{F\in\fT}$.
\end{proposition}

\begin{proof}
  Assume that we have local extension operators $E^\cT_K$ for every $K\in\fT$, and let $E^F_K:=\tr^\cT_FE^\cT_K$. We check the three properties of a consistent family of extension operators. First, we have $E_K^K=\tr_K^\cT E_K^\cT$, which is the inclusion. Next, if $K\not\le F$, then we have $E^F_K=\tr^{\cT}_FE^{\cT}_K=0$. Finally, if $G\le F$, then $\tr_G^FE_K^F=\tr_G^F\tr_F^\cT E^\cT_K=\tr_G^\cT E^\cT_K=E^G_K$.

  Conversely, assume that we have a consistent family of extension operators, and define $E_K^\cT:=\left(E_K^F\right)_{F\in\fT}$. A priori, the codomain of this operator is $\prod_{F\in\fT}\cF(F)$; we check that the image is indeed in $\cF(\cT)$ because for $G\le F$ we have $\tr_G^FE_K^F=E_K^G$. Then, by definition, $\tr^\cT_KE^\cT_K=E_K^K$, which is the inclusion. Likewise, if $K\not\le F$, then $\tr^\cT_FE^\cT_K=E^F_K=0$.
\end{proof}

As a result, we can treat local extension operators and consistent families of extension operators interchangeably. Moreover, we can view them together as a consistent family of extension operators on the extended partially ordered set $\hat\fT:=\fT\sqcup\{\cT\}$.

\begin{proposition}\label{prop:consistenthat}
  Let $\cF$ be a function space on a partially ordered set $\fT$, and let $\cF(\cT)$ be the corresponding global function space. Assume that we have local extension operators $E_K^\cT$ for every $K\in\fT$, and correspondingly a consistent family of extension operators $E_K^F$ on $\fT$. Then the operators $E_K^\cT$ and $E_K^F$ together yield a consistent family of extension operators on the partially ordered set $\hat\fT=\fT\sqcup\{\cT\}$.
\end{proposition}

\begin{proof}
  We first complete the definition by noting that $E_\cT^F=0$ for any $F\in\hat\fT$ because the domain of $E_\cT^F$ is $\oF(\cT)=0$. We now check that the three conditions in Definition~\ref{def:consistent} still hold if one of the elements is $\cT$. We vacuously have that $E_\cT^\cT\colon0\to\cF(\cT)$ is the inclusion. For the second property, if $K=\cT$ then $E_\cT^F=0$ as just observed, and if $F=\cT$, then we have $K\le\cT$, so the condition is vacuous. For the third property, if $K=\cT$, then both sides are zero. If $F=\cT$, then we must check $\tr_G^\cT E_K^\cT=E_K^G$. For $G\neq\cT$, this equation is just the equation by which local extension operators yield a consistent family of extension operators in Proposition~\ref{prop:localconsistent}. If $G=\cT$, the equation is a tautology because $\tr_\cT^\cT$ is the identity.
\end{proof}

The conclusion of these results is that we do not lose anything by assuming that our partially ordered sets have a top element: If we have a partially ordered $\fT$ without a top element, we can add in a top element $\cT$ to obtain $\hat\fT:=\fT\sqcup\{\cT\}$. As we have shown, a function space on $\fT$ yields a function space on $\hat\fT$, and a consistent family of extension operators on $\fT$ yields a consistent family of extension operators on $\hat\fT$.

\subsection{Geometric decomposition}\label{subsec:geodecomp}
We now set out to prove the geometric decomposition theorem. Our proof is similar to the proof in \cite[Theorem~4.3]{afw09}, but in a much more general context. The more general setting lets us simplify the proof in some ways, inducting on elements instead of on dimension. In particular, we do not even need to have a notion of dimension for elements of $\hat\fT$, allowing us to include more abstract elements like the triangulation $\cT$ itself.

In what follows, we will work with a function space on a partially ordered set $\hat\fT$ with top element denoted $\cT$, but we need not assume that $\cF(\cT)$ is the global function space as defined above, and hence we need not assume that $\oF(\cT)=0$. 

\begin{definition}\label{def:ci}
  Let $\hat\fT$ be a partially ordered set with a top element denoted $\cT$, and let $\cF$ be a function space on $\hat\fT$. Let $\fS$ be a lower subset of $\hat\fT$, that is, if $F\in\fS$ and $K\le F$, then $K\in\fS$. Let $\cI(\fS)$ denote those elements of $\cF(\cT)$ that vanish on $\fS$, that is,
  \begin{equation*}
    \cI(\fS):=\Set{\phi\in\cF(\cT)|\tr^\cT_F\phi=0\text{ for all }F\in\fS}.
  \end{equation*}
\end{definition}

\begin{lemma}\label{lem:induct}
    Let $\hat\fT$ be a partially ordered set with a top element $\cT$, and let $\cF$ be a function space on $\hat\fT$. Assume that we have a consistent family of extension operators on $\hat\fT$. Let $\fS$ be a lower subset of $\hat\fT$, and let $T$ be a maximal element of $\fS$, so $\fS\setminus\{T\}$ is also a lower subset of $\hat\fT$. Then we have
  \begin{equation*}
    \cI(\fS\setminus\{T\})=\cI(\fS)\oplus E_T^\cT\oF(T).
  \end{equation*}
\end{lemma}

\begin{proof}
  First, we have that $\cI(\fS)\subseteq\cI(\fS\setminus\{T\})$ by definition. Meanwhile, we have $E_T^\cT\oF(T)\subseteq\cI(\fS\setminus\{T\})$ because $T$ is a maximal element, so if $F\in\fS\setminus\{T\}$ then $T\not\le F$, and so $\tr_F^\cT E_T^\cT=E_T^F=0$.

  Next, assume that $\phi$ is in both $\cI(\fS)$ and $E_T^\cT\oF(T)$. Then, $\tr^\cT_T\phi=0$ because $T\in\fS$, and $\phi=E_T^\cT\psi$ for some $\psi\in\oF(T)$. We conclude that $0=\tr^\cT_TE_T^\cT\psi=E^T_T\psi=\psi$, so $\phi=0$.

  Finally, let $\phi\in\cI(\fS\setminus\{T\})$. Let $\psi=\tr^\cT_T\phi$. Observe that $\psi\in\oF(T)$ because if $F<T$ then $F\in\fS\setminus\{T\}$ so $\tr^T_F\psi=\tr^\cT_F\phi=0$. So now let $\phi'=\phi-E^\cT_T\psi$, and it remains to show that $\phi'\in\cI(\fS)$. Since we already know that $\phi'\in\cI(\fS\setminus\{T\})$, we just need to observe that $\tr_T^\cT\phi'=\tr_T^\cT\phi-\tr_T^\cT E^\cT_T\psi=\psi-\psi=0$.
\end{proof}

\begin{theorem}\label{thm:decomp}
    Let $\hat\fT$ be a finite partially ordered set with a top element $\cT$, and let $\cF$ be a function space on $\hat\fT$. Assume that we have a consistent family of extension operators on $\hat\fT$. Then
  \begin{equation*}
    \cF(\cT)=\bigoplus_{F\in\hat\fT}E_F^\cT\oF(F).
  \end{equation*}
\end{theorem}

\begin{proof}
  We induct on the statement that
  \begin{equation*}
    \cI(\fS)=\bigoplus_{F\notin\fS}E_F^\cT\oF(F)
  \end{equation*}
  for all lower sets $\fS$. The base case is $\fS=\hat\fT$, in which case $\cI(\fS)=0$ because the condition $\tr_\cT^\cT\phi=0$ yields $\phi=0$. Meanwhile, $\bigoplus_{F\notin\fS}$ is an empty sum. Our desired claim is given by $\fS=\emptyset$, so we induct on the number of elements not in $\fS$.

  Given a lower set $\fS'$ that is not all of $\hat\fT$, because $\hat\fT$ is finite, we can find a minimal element $T$ in the complement of $\fS'$. Then $\fS:=\fS'\sqcup\{T\}$ is also a lower set, with more elements. The inductive hypothesis tells us that $\cI(\fS)=\bigoplus_{F\notin\fS}E_F^\cT\oF(F)$, and Lemma~\ref{lem:induct} tells us that $\cI(\fS')=\cI(\fS)\oplus E_T^\cT\oF(T)$, so $\cI(\fS')=\bigoplus_{F\notin\fS'}E_F^\cT\oF(F)$, as desired.
\end{proof}

\begin{corollary}\label{cor:globaldecomp}
  Let $\cF$ be a function space on a finite partially ordered set $\fT$, and let $\cF(\cT)$ be the corresponding global function space. Assume that we have local extension operators $E^\cT_F\colon\oF(F)\to\cF(\cT)$ for every $F\in\fT$. Then
  \begin{equation*}
    \cF(\cT)=\bigoplus_{F\in\fT}E_F^\cT\oF(F).
  \end{equation*}
\end{corollary}

\begin{proof}
  We obtain a function space and a consistent family of extension operators on $\hat\fT:=\fT\sqcup\{\cT\}$ as before (Propositions~\ref{prop:functionhat} and \ref{prop:consistenthat}). The result follows from Theorem~\ref{thm:decomp} and the fact that $\oF(\cT)=0$.
\end{proof}

\begin{corollary}\label{cor:localdecomp}
  Let $\cF$ be a function space on a finite partially ordered set $\fT$. Assume that we have local extension operators $E^\cT_F\colon\oF(F)\to\cF(\cT)$ for every $F\in\fT$. Then, for every $T\in\fT$, we have
  \begin{equation*}
    \cF(T)=\bigoplus_{F\le T}E_F^T\oF(F).
  \end{equation*}
\end{corollary}

\begin{proof}
  Let $\Delta(T)=\Set{F\in\fT| F\le T}$. Then $\Delta(T)$ is a finite partially ordered set with top element $T$, and $\cF$ restricted to $\Delta(T)$ is a function space. By Proposition~\ref{prop:localconsistent}, we have a consistent family of extension operators on $\fT$, which we can restrict to a consistent family of extension operators on $\Delta(T)$. The claim then follows from applying Theorem~\ref{thm:decomp} to $\Delta(T)$.
\end{proof}

For readers interested in the comparison between our definition of consistent extension operators and the definition in \cite[Equation~(4.4)]{afw09}, we now show that, via the geometric decomposition, a family of consistent extension operators in the sense of this paper yields a family of consistent extension operators in the sense of \cite{afw09}.

\begin{corollary}\label{cor:afw}
  Let $\cF$ be a function space on a finite partially ordered set $\fT$, and assume that we have a consistent family of extension operators $E_G^F\colon\oF(G)\to\cF(F)$. Then we can extend $E_G^F$ to a linear map on the full space $\cF(G)$ by defining
  \begin{equation}\label{eq:extendextension}
    E_G^F(E_K^G\psi):=E_K^F\psi
  \end{equation}
  for every $K\le G$ and $\psi\in\oF(K)$.

  These extension operators satisfy the properties that
  \begin{itemize}
  \item $\tr^F_GE^F_G\colon\cF(G)\to\cF(G)$ is the identity whenever $G\le F$,
  \item $\tr_G^TE_F^T=E^G_K\tr_K^F\colon\cF(F)\to\cF(G)$ whenever $K\le F,G\le T$.
  \end{itemize}
\end{corollary}

\begin{proof}
  We first check that map $E_G^F$ on the full space $\cF(G)$ is well-defined. Since $\cF(G)=\bigoplus_{K\le G}E_K^G\oF(K)$, defining an operator on $\cF(G)$ is equivalent to defining the operator on each summand $E_K^G\oF(K)$. This definition extends the original definition because, if $K=G$, then $E^G_K\psi=\psi$, so both sides of \eqref{eq:extendextension} are the same.

  We now verify the claimed properties. For $G\le F$, we have $\tr^F_GE^F_G(E_K^G\psi)=\tr^F_GE_K^F\psi=E_K^G\psi$ for every $K\le G$ and $\psi\in\oF(K)$.

  For $K\le F,G\le T$, we must check $\tr^T_GE^T_FE_L^F\psi=E^G_K\tr_K^FE_L^F\psi$ for every $L\le F$ and $\psi\in\oF(L)$. The left-hand side is $\tr_G^TE_F^TE_L^F\psi=\tr_G^TE_L^T\psi=E_L^G\psi$. The right-hand side is $E_K^G\tr_K^FE_L^F\psi=E_K^GE_L^K\psi=E_L^G\psi$.
\end{proof}

\subsection{Dual decomposition}
We now prove the dual version of Theorem~\ref{thm:decomp} for the degrees of freedom, that is, a decomposition of the dual space $\cF(\cT)^*$ in terms of images of $\tr^*$.

We will need one additional ingredient. Observe that the inclusion $\oF(F)\to\cF(F)$ yields a corresponding surjection $\cF(F)^*\to\oF(F)^*$. We will need to arbitrarily choose a complement of the kernel of this map, which we will denote $\oF(F)^\dagger$. In other words, $\oF(F)^\dagger$ is a subspace of $\cF(F)^*$ such that the restriction of the surjection $\cF(F)^*\to\oF(F)^*$ to $\oF(F)^\dagger\to\oF(F)^*$ is an isomorphism. We express this property more explicitly in the following definition.

\begin{definition}\label{def:dagger}
  Let $\cF$ be a function space on a partially ordered set $\fT$. For $F\in\fT$, let $\oF(F)^\dagger$ denote a subspace of $\cF(F)^*$ satisfying the properties that
  \begin{itemize}
  \item if $\alpha\in\oF(F)^*$, there exists an $\beta\in\oF(F)^\dagger$ such that $\beta(\phi)=\alpha(\phi)$ for every $\phi\in\oF(F)$, and
  \item if $\beta\in\oF(F)^\dagger$ and $\beta(\phi)=0$ for all $\phi\in\oF(F)$, then $\beta=0$.
  \end{itemize}
\end{definition}

In the context of finite element exterior calculus \cite{afw09}, the spaces $\oF(F)^\dagger$ are essentially the $W$ spaces in \cite[Theorem 5.1]{afw09}. More specifically, $W(T,F)$ in \cite{afw09} is $\left(\tr_F^T\right)^*\oF(F)^\dagger$ in the language of this paper. With this choice of $\oF(F)^\dagger$, we obtain the same dual decomposition as in \cite{afw09}.

In the general context, if $\cF(F)$ is a Hilbert space, then one option is to simply take $\oF(F)^\dagger$ to be the linear functionals on $\cF(F)$ given by pairing with elements of $\oF(F)$; see also \cite[Proposition~2.5]{chra16}. Note also that, if $\cF(F)$ is only a Banach space, then such a complement of the kernel need not exist. Regardless, if we have extension operators, then another option is to use the geometric decomposition to produce a projection $\cF(F)\to\oF(F)$ and then take $\oF(F)^\dagger$ to be the image of the dual map $\oF(F)^*\to\cF(F)^*$. We express these ideas in the following two propositions.

\begin{proposition}
  Let $\cF$ be a function space on a partially ordered set $\fT$. Assume that each $\cF(F)$ is a Hilbert space. Let $\oF(F)^\dagger$ be the set of all linear functionals on $\cF(F)$ of the form $\langle\psi,\cdot\rangle$, where $\psi\in\oF(F)$. Then $\oF(F)^\dagger$ satisfies Definition~\ref{def:dagger}.
\end{proposition}

\begin{proof}
  Since $\oF(F)$ is a closed subspace of $\cF(F)$, it is also a Hilbert space. Hence, if $\alpha\in\oF(F)^*$, then, by the Riesz representation theorem, there exists a $\psi\in\oF(F)$ such that $\alpha(\phi)=\langle\psi,\phi\rangle$ for every $\phi\in\oF(F)$. Next, if $\psi\in\oF(F)$ and $\langle\psi,\phi\rangle=0$ for all $\phi\in\oF(F)$, then $\psi=0$.
\end{proof}

\begin{proposition}
  Let $\cF$ be a function space on a partially ordered set $\fT$. Assume that we have local extension operators $E_F^\cT\colon\oF(F)\to\cF(\cT)$ for every $F\in\fT$, so we have a geometric decomposition of $\cF(F)$ by Corollary~\ref{cor:localdecomp}, which, in particular, yields a projection $\pi\colon\cF(F)\to E_F^F\oF(F)=\oF(F)$. Let $\oF(F)^\dagger$ be the image of the dual map $\pi^*\colon\oF(F)^*\to\cF(F)^*$. Then $\oF(F)^\dagger$ satisfies Definition~\ref{def:dagger}.
\end{proposition}

\begin{proof}
  If $\alpha\in\oF(F)^*$, let $\beta=\pi^*\alpha\in\oF(F)^\dagger$, so $\beta(\phi)=\alpha(\pi\phi)=\alpha(\phi)$ for every $\phi\in\oF(F)$ because $\pi\phi=\phi$. Next, if $\beta=\pi^*\alpha$ for $\alpha\in\oF(F)^*$ and $\beta(\phi)=0$ for all $\phi\in\oF(F)$, then $\alpha(\phi)=\alpha(\pi\phi)=\beta(\phi)=0$ for all $\phi\in\oF(F)$, so $\alpha=0$ and hence $\beta=0$.
\end{proof}

We now proceed with the decomposition of the dual space, following very similar steps to the proof of Theorem~\ref{thm:decomp}, except that, as we will see, the induction goes the other way.

\begin{definition}
  Let $\hat\fT$ be a partially ordered set with a top element denoted $\cT$, and let $\cF$ be a function space on $\hat\fT$. Let $\fS$ be a lower subset of $\hat\fT$. Let $\cC(\fS)$ denote those functionals in the dual $\cF(\cT)^*$ that vanish on $\cI(\fS)$ (see Definition~\ref{def:ci}), that is,
  \begin{equation*}
    \cC(\fS):=\Set{\alpha\in\cF(\cT)^*|\alpha(\phi)=0\text{ for all }\phi\in\cI(\fS)}.
  \end{equation*}
\end{definition}

\begin{lemma}\label{lem:inductdual}
  Let $\hat\fT$ be a partially ordered set with top element $\cT$, and let $\cF$ be a function space on $\hat\fT$, and assume that we have chosen spaces $\oF(F)^\dagger$ as defined above. Assume that we have a consistent family of extension operators on $\hat\fT$. Let $\fS$ be a lower subset of $\hat\fT$, and let $T$ be a maximal element of $\fS$. Then we have
  \begin{equation*}
    \cC(\fS)=\cC(\fS\setminus\{T\})\oplus\left(\tr^\cT_T\right)^*\oF(T)^\dagger.
  \end{equation*}  
\end{lemma}

\begin{proof}
  Since $\cI(\fS)\subseteq\cI(\fS\setminus\{T\})$, we have $\cC(\fS\setminus\{T\})\subseteq\cC(\fS)$. Next, in order to show that $\left(\tr_T^\cT\right)^*\oF(T)^\dagger\subseteq\cC(\fS)$, let $\alpha\in\oF(T)^\dagger$ and $\phi\in\cI(\fS)$. Then $\left((\tr^\cT_T)^*\alpha\right)(\phi)=\alpha(\tr^\cT_T\phi)=0$ because $T\in\fS$.

  Now, assume $\alpha$ is in both $\cC(\fS\setminus\{T\})$ and $\left(\tr^\cT_T\right)^*\oF(T)^\dagger$, and we aim to show $\alpha=0$. We have $\alpha(\phi)=0$ for all $\phi\in\cI(\fS\setminus\{T\})$, and $\alpha=\left(\tr^\cT_T\right)^*\beta$ for some $\beta\in\oF(T)^\dagger$, so we aim to show that $\beta=0$. It suffices to show that $\beta(\psi)=0$ for all $\psi\in\oF(T)$. Let $\phi=E_T^\cT\psi$, so $\phi\in\cI(\fS\setminus\{T\})$ by Lemma~\ref{lem:induct}. Thus, $\alpha(\phi)=0$. On the other hand, $\alpha(\phi)=\beta(\tr_T^\cT E_T^\cT\psi)=\beta(\psi)$. Hence, $\beta(\psi)=0$, as desired.

  Finally, let $\alpha\in\cC(\fS)$. We aim to write $\alpha$ in the form $\alpha=\alpha'+\left(\tr_T^\cT\right)^*\beta$, where $\alpha'\in\cC(\fS\setminus\{T\})$ and $\beta\in\oF(T)^\dagger$. Since $E_T^\cT\colon\oF(T)\to\cF(\cT)$, we have a corresponding dual map $\left(E_T^\cT\right)^*\colon\cF(\cT)^*\to\oF(T)^*$, and $\oF(T)^*$ is isomorphic to $\oF(T)^\dagger$. Let $\beta\in\oF(T)^\dagger$ be the image of $\left(E_T^\cT\right)^*\alpha$ under this isomorphism, that is, $\beta(\psi)=\alpha(E_T^\cT\psi)$ for every $\psi\in\oF(T)$. Let $\alpha'=\alpha-\left(\tr_T^\cT\right)^*\beta$. Since we have shown that $\left(\tr_T^\cT\right)^*\oF(T)^\dagger\subseteq\cC(\fS)$, we know that $\alpha'\in\cC(\fS)$. We must show that $\alpha'\in\cC(\fS\setminus\{T\})$, that is, that $\alpha'$ vanishes on $\cI(\fS\setminus\{T\})$. Recalling from Lemma~\ref{lem:induct} that $\cI(\fS\setminus\{T\})=\cI(\fS)\oplus E_T^\cT\oF(T)$, since we already know that $\alpha'$ vanishes on $\cI(\fS)$, it suffices to show that $\alpha'$ vanishes on $E_T^\cT\psi$ for all $\psi\in\oF(T)$. We do so by computing that $\left((\tr_T^\cT\right)^*\beta)(E_T^\cT\psi)=\beta(\tr_T^\cT E_T^\cT\psi)=\beta(\psi)=\alpha(E_T^\cT\psi)$, so $\alpha'(E_T^\cT\psi)=\alpha(E_T^\cT\psi)-\alpha(E_T^\cT\psi)=0$.
\end{proof}

\begin{theorem}\label{thm:dualdecomp}
  Let $\hat\fT$ be a finite partially ordered set with a top element $\cT$, and let $\cF$ be a function space on $\hat\fT$. Assume that we have chosen subspaces $\oF(F)^\dagger$ of $\cF(F)^*$ that are isomorphic to $\oF(F)^*$ as defined in Definition~\ref{def:dagger}. Assume that we have a consistent family of extension operators on $\hat\fT$. Then
  \begin{equation*}
    \cF(\cT)^*=\bigoplus_{F\in\hat\fT}\left(\tr^\cT_F\right)^*\oF(F)^\dagger.
  \end{equation*}
\end{theorem}

\begin{proof}
  We induct on the statement that
  \begin{equation*}
    \cC(\fS)=\bigoplus_{F\in\fS}\left(\tr^\cT_F\right)^*\oF(F)^\dagger.
  \end{equation*}
  for all lower sets $\fS$. The base case is $\fS=\emptyset$, for then $\cI(\fS)=\cF(\cT)$, so $\cC(\fS)=0$. Meanwhile, the right-hand side is an empty sum. Our desired claim is given by $\fS=\hat\fT$, for then $\cI(\fS)=0$, so $\cC(\fS)=\cF(\cT)^*$. So, we induct on the number of elements in $\fS$.

  Given a nonempty lower set $\fS$, because $\fS$ is finite, we can find a maximal element of $\fS$. Then $\fS\setminus\{T\}$ is also a lower set, with fewer elements, so the inductive hypothesis tells us that $\cC(\fS\setminus\{T\})=\bigoplus_{F\in\fS\setminus\{T\}}\left(\tr^\cT_F\right)^*\oF(F)^\dagger$, and Lemma~\ref{lem:inductdual} tells us that $\cC(\fS)=\cC(\fS\setminus\{T\})\oplus\left(\tr^\cT_T\right)^*\oF(T)^\dagger$, so $\cC(\fS)=\bigoplus_{F\in\fS}\left(\tr^\cT_F\right)^*\oF(F)^\dagger$, as desired.
\end{proof}

We now state corollaries analogous to Corollaries~\ref{cor:globaldecomp} and \ref{cor:localdecomp}. The proofs are also analogous, so we present abbreviated versions.

\begin{corollary}\label{cor:dualdecomp}
  Let $\cF$ be a function space on a finite partially ordered set $\fT$, and let $\cF(\cT)$ be the corresponding global function space. Assume that we have chosen subspaces $\oF(F)^\dagger$ of $\cF(F)^*$ that are isomorphic to $\oF(F)^*$, and assume that we have local extension operators $E_F^\cT\colon\oF(F)\to\cF(\cT)$ for every $F\in\fT$. Then
  \begin{equation*}
    \cF(\cT)^*=\bigoplus_{F\in\fT}\left(\tr^\cT_F\right)^*\oF(F)^\dagger.    
  \end{equation*}
\end{corollary}

\begin{proof}
  We define a function space on $\hat\fT:=\fT\sqcup\{\cT\}$ as before, and then apply Theorem~\ref{thm:dualdecomp} to $\hat\fT$, noting that $\oF(\cT)$ is zero, so $\oF(\cT)^\dagger$ is zero as well.
\end{proof}

\begin{corollary}
  Let $\cF$ be a function space on a finite partially ordered set $\fT$. Assume that we have chosen subspaces $\oF(F)^\dagger$ of $\cF(F)^*$ that are isomorphic to $\oF(F)^*$, and assume that we have local extension operators $E_F^\cT\colon\oF(F)\to\cF(\cT)$ for every $F\in\fT$. Then, for every $T\in\fT$, we have
  \begin{equation*}
    \cF(T)^*=\bigoplus_{F\le T}\left(\tr^\cT_F\right)^*\oF(F)^\dagger.    
  \end{equation*}
\end{corollary}

\begin{proof}
  We apply Theorem~\ref{thm:dualdecomp} to $\Delta(T)$.
\end{proof}

\section{Simplicial function spaces and triangulations}\label{sec:simplicial}
We now specialize to simplicial meshes $\cT$. We remark that the context can still be quite general: $\cT$ need not be a manifold, could have parts of different dimension, etc.

We begin by defining a simplicial function space, which is defined independently of a mesh. While the definition is quite general, a good example to keep in mind is the finite element exterior calculus space $\cF=\mathcal P_r\Lambda^k$, or, more generally, any affine-invariant space of covariant tensors.

\begin{definition}
  A \emph{simplicial map} between two simplices $K$ and $F$ is an affine function $\Phi\colon K\to F$ that sends vertices to vertices.
\end{definition}

Note that simplicial maps are allowed to send two vertices of $K$ to the same vertex of $F$; they need not be inclusions.

\begin{definition}
  A \emph{simplicial function space} $\cF$ is
  \begin{itemize}
  \item a vector space $\cF(F)$ associated to every simplex $F$, and
  \item a linear map $\Phi^*\colon\cF(F)\to\cF(K)$, called the \emph{pullback}, associated to every simplicial map $\Phi\colon K\to F$,
  \end{itemize}
  satisfying the functoriality properties that
  \begin{itemize}
  \item if $\Phi$ is the identity map on $F$, then $\Phi^*$ is the identity map on $\cF(F)$, and
  \item $(\Psi\circ\Phi)^*=\Phi^*\Psi^*$ whenever we have simplicial maps $\Phi\colon K\to G$ and $\Psi\colon G\to F$.
  \end{itemize}
\end{definition}

In the context of a simplicial mesh, a simplicial function space yields a function space in the sense of Definition~\ref{def:funspace}.

\begin{proposition}
  Let $\cF$ be a simplicial function space, and let $\cT$ be a simplicial complex. Let $\fT$ be the partially ordered set of the vertices, edges, faces, etc., of $\cT$, ordered by inclusion. Then $\cF$ is a function space on $\fT$ in the sense of Definition~\ref{def:funspace}, where the trace maps $\tr_K^F$ are the pullbacks via the inclusion maps $K\hookrightarrow F$.
\end{proposition}

\begin{proof}
  The claim follows by comparing the two definitions.
\end{proof}

Consequently, we can use all of the notation of the previous section, such as the vanishing trace subspace $\oF(F)$ (Definition~\ref{def:vanishingtrace}), the global function space $\cF(\cT)$ (Definition~\ref{def:globalfunctionspace}), and so forth.

We now extend the notion of simplicial map to triangulations.

\begin{definition}
  Let $\cT$ be a simplicial complex, and let $Q$ be a simplex. A \emph{simplicial map} $\Phi\colon\cT\to Q$ is a piecewise affine map that sends vertices of $\cT$ to vertices of $Q$. By piecewise affine, we mean that $\Phi$ is affine when restricted to any $F\in\fT$.
\end{definition}

Note that, in general, simplicial maps can be defined with a simplicial complex codomain as well, but we will not need this notion in this paper. We now define pullback.

\begin{definition}
  Given a simplicial map $\Phi\colon\cT\to Q$, define $\Phi^*\colon\cF(Q)\to\cF(\cT)$ piecewise, that is, for $\phi\in\cF(Q)$, let
  \begin{equation}\label{eq:pullback}
    \Phi^*\phi:=\left(\left(\Phi\rvert_F\right)^*\phi\right)_{F\in\fT}.
  \end{equation}
\end{definition}

Of course, we need to check that $\Phi^*\phi$ as defined above yields an element of $\cF(\cT)$.

\begin{proposition}
  Equation~\eqref{eq:pullback} defines an element of $\cF(\cT)$ as defined in Definition~\ref{def:globalfunctionspace}.
\end{proposition}

\begin{proof}
  We must check that $\tr_G^F\left(\Phi\rvert_F\right)^*\phi=\left(\Phi\rvert_G\right)^*\phi$ for $G\le F$, which follows from the fact that $\Phi\rvert_G=\Phi\rvert_F\circ I_G^F$, where $I_G^F$ is the inclusion of $G$ into $F$. Taking pullbacks of both sides and recalling that the trace is the pullback via the inclusion, we obtain our claim.
\end{proof}

We now define two types of standard simplices.

\begin{definition}
  Let $\R^{n+1}$ have coordinates $(\lambda_0,\dotsc,\lambda_n)$, and let $e_0,\dotsc,e_n$ denote the unit coordinate vectors. Let $T^n$ be the simplex with vertices $e_0,\dotsc,e_n$, and let $Q^{n+1}$ be the simplex with vertices $0,e_0,\dotsc,e_n$. In other words,
  \begin{align*}
    T^n&=\Set{(\lambda_0,\dotsc,\lambda_n)|\lambda_i\ge0,\ \lambda_0+\dots+\lambda_n=1},\\
    Q^{n+1}&=\Set{(\lambda_0,\dotsc,\lambda_n)|\lambda_i\ge0,\ \lambda_0+\dots+\lambda_n\le1}.
  \end{align*}
\end{definition}

Note that $T^n$ is a face of $Q^{n+1}$. Since any two simplices of the same dimension are isomorphic, it suffices to study a simplicial function spaces on the $T^n$ (or on the $Q^{n+1}$), so we define extension operators in this context.

\begin{definition}\label{def:simpext}
  Let $\cF$ be a simplicial function space. A \emph{simplicial extension operator} in dimension $n$ is an operator $E_n\colon\oF(T^n)\to\cF(Q^{n+1})$ such that
  \begin{itemize}
  \item $\tr_{T^n}^{Q^{n+1}}E_n\colon\oF(T^n)\to\cF(T^n)$ is the inclusion, and
  \item $\tr_S^{Q^{n+1}}E_n=0$ if $T^n\not\le S$, that is, if $S$ is a face of $Q^{n+1}$ other than $T^n$ or $Q^{n+1}$ itself.
  \end{itemize}
\end{definition}
In other words, as with the other extension operators in this paper, the extensions $E_n\psi$ are required to vanish everywhere they can.

Our main result is that, in the context of a triangulation, the existence of simplicial extension operators is equivalent to the existence of local extension operators, and hence equivalent to the existence of geometric decompositions. Consequently, we can obtain all of the results of this paper for all simplicial meshes at once simply by constructing an $E_n$ for each $n$.
\begin{theorem}\label{thm:simpext}
  Let $\cF$ be a simplicial function space, and let $\cT$ be a simplicial complex, with maximal dimension $n$. Let $\fT$ be the partially ordered set of the vertices, edges, faces, etc., of $\cT$. Then $\fT$ has local extension operators $E_K^\cT$ if and only if $\cF$ has simplicial extension operators $E_m$ for every $m<n$.
\end{theorem}

\begin{proof}
  Assume that we have simplicial extension operators $E_m$ for every $m<n$. Let $K$ be a face of $\cT$ of dimension $m<n$, and let $v_0,\dotsc,v_m$ be the vertices of $K$. Given $K$ and an ordering of its vertices, there is a natural map from $\cT$ to $Q^{m+1}$, which will be key to our construction. Specifically, let $\Phi_K^\cT\colon\cT\to Q^{m+1}$ be the piecewise affine map that sends each vertex $v_i$ of $K$ to the corresponding vertex $e_i$ of $Q^{m+1}$, and sends every other vertex of the triangulation to the vertex $0$ of $Q^{m+1}$. Equivalently, for every point $x$, we have
  \begin{equation*}
    (\lambda_0,\dotsc,\lambda_m)=\Phi_K^\cT(x)=(\lambda_{v_0}(x),\dotsc,\lambda_{v_m}(x)),
  \end{equation*}
  where, for every vertex $v$ of the triangulation, $\lambda_v$ denotes the continuous piecewise linear real-valued function on $\cT$ that is equal to one at $v$ and equal to zero at every other vertex of $\cT$.

  The idea of the construction is that, given $\phi\in\oF(K)$, using the isomorphism between the $m$-simplex $K$ and the standard $m$-simplex $T^m$, we obtain a corresponding form $\psi\in\oF(T^m)$, which we can extend to $E_m\psi\in\cF(Q^{m+1})$. Then, we simply pull back $E_m\psi$ via $\Phi_K^\cT$ to obtain the desired local extension operator $E_K^\cT\phi\in\cF(\cT)$. We now discuss each of these steps in detail.

  First, we establish some notation for the various restrictions of the map $\Phi_K^\cT$. For any face $F$ of the triangulation, let $\Phi_K^F\colon F\to Q^{m+1}$ be the restriction of $\Phi_K^\cT$ to $F$, let $Q_K^F\le Q^{m+1}$ be the image of $\Phi_K^F$, and let $\Psi_K^F\colon F\to Q_K^F$ be the same function as $\Phi_K^F$, but with restricted codomain, so $\Phi_K^F=I_{Q_K^F}^{Q^{m+1}}\circ\Psi_K^F$, where $I_{Q_K^F}^{Q^{m+1}}\colon Q_K^F\hookrightarrow Q^{m+1}$ is the inclusion. In other words, we have the diagram
  \begin{equation}\label{eq:diagrampsi}
    \begin{tikzcd}
      \cT\ar[r,"\Phi_K^\cT"]&Q^{m+1}\\
      F\ar[twoheadrightarrow,r,"\Psi_K^F"]\ar[hookrightarrow,u]&Q_K^F\ar[hookrightarrow,u]
    \end{tikzcd}
  \end{equation}

  Observe that if $F=K$, then $\Phi_K^K$ sends each vertex $v_i$ of $K$ to the corresponding unit vector $e_i$, so $Q_K^K=T^m$, and $\Psi_K^K\colon K\to T^m$ is an isomorphism. Next, observe that if $K\not\le F$, then there is vertex $v_i$ of $K$ that is not a vertex of $F$. Consequently, $e_i$ is not a vertex of $Q_K^F$, so $T^m\not\le Q_K^F$. As we will see, these observations correspond to the two properties of local extension operators in Definition~\ref{def:local}.

  By functoriality of $\cF$, the fact that $\Psi_K^K\colon K\to T^m$ is an isomorphism of simplices implies that the pullback $\left(\Psi_K^K\right)^*\colon\cF(T^m)\to\cF(K)$ is an isomorphism of vector spaces, and also that $\left(\Psi_K^K\right)^*\colon\oF(T^m)\to\oF(K)$ is an isomorphism between the corresponding spaces of vanishing trace. So, we now define our local extension maps $E_K^\cT$ via the following diagram.
  \begin{equation}\label{eq:diagramext}
    \begin{tikzcd}
      \cF(Q^{m+1})\ar[r, "\left(\Phi_K^\cT\right)^*"]&\cF(\cT)\\
      \oF(T^m)\ar[u, "E_m"]\ar[r, "\left(\Psi_K^K\right)^*", "\cong"']&\oF(K)\ar[u,"E_K^\cT"]
    \end{tikzcd}
  \end{equation}
  In other words, we define $E_K^\cT:=\left(\Phi_K^\cT\right)^*\circ E_m\circ\left(\left(\Psi_K^K\right)^*\right)^{-1}$.

  It remains to check that the maps $E_K^\cT$ satisfy the properties of local extension operators in Definition~\ref{def:local}. To do so, we apply $\cF$ to Diagram~\eqref{eq:diagrampsi} and combine it with we Diagram~\eqref{eq:diagramext} to obtain
  \begin{equation*}
    \begin{tikzcd}
      \cF(Q^F_K)\ar[r,"\left(\Psi_K^F\right)^*"]&\cF(F)\\
      \cF(Q^{m+1})\ar[r, "\left(\Phi_K^\cT\right)^*"]\ar[u,"\tr^{Q^{m+1}}_{Q^F_K}"]&\cF(\cT)\ar[u,"\tr^\cT_F"]\\
      \oF(T^m)\ar[u, "E_m"]\ar[r, "\left(\Psi_K^K\right)^*", "\cong"']&\oF(K)\ar[u,"E_K^\cT"]
    \end{tikzcd}    
  \end{equation*}

  If $F=K$, so $Q_K^F=T^m$, then the composition in the left column is $\tr^{Q^{m+1}}_{T^m}E_m\colon\oF(T^m)\to\cF(T^m)$, which is the inclusion by Definition~\ref{def:simpext}. The top and bottom maps of the diagram are both the isomorphism $\left(\Psi_K^K\right)^*$, so the composition in the right column, $\tr^\cT_FE_K^\cT\colon\oF(K)\to\cF(F)$, is also the inclusion, as desired.

  Meanwhile, if $K\not\le F$, then we have observed that $T^m\not\le Q_K^F$, so the composition in the left column, $\tr_{Q^F_K}^{Q^{m+1}}E_m$, is zero by Definition~\ref{def:simpext}. Since the diagram commutes and the bottom map $\left(\Psi_K^K\right)^*$ is an isomorphism, we conclude that the composition in the right column, $\tr_F^\cT E_K^\cT$, is zero as well, as desired.

  Recall that, so far, we assumed that $K$ has dimension $m<n$. It remains then to consider the case where $K$ has dimension $m=n$, the maximal dimension of the triangulation. In this case, we simply extend by zero, that is, we set $E_K^K$ to be the inclusion, and $E_K^F$ to be zero for all other faces $F$. So, the first two properties of Definition~\ref{def:consistent} follow immediately. For the third property, we assume $G\le F$, and show that $\tr_G^FE_K^F=E_K^G$. If neither $F$ nor $G$ are $K$, then both sides of the equation are zero. If $G=K$, then $F$ must also be $K$ because $K$ has maximal dimension; in this case, both sides of the equation are the inclusion. Finally, if $F=K$ and $G$ is not, then for all $\psi\in\oF(K)$, we have $\tr_G^KE_K^K\psi=\tr_G^K\psi=0$ because $\psi$ has vanishing trace, and $E_K^G=0$ as just defined. Per Proposition~\ref{prop:localconsistent}, the corresponding operator $E_K^\cT$ is a local extension operator, as desired.

  Finally, it remains to prove the converse, so we assume that we have local extension operators and hence consistent extension operators $E_K^F$. We will construct the simplicial extension operators $E_m$ for every $m<n$. We have all of the ingredients already. Let $T$ be a face of the triangulation of dimension $m+1$, and let $K$ be a subface of $T$ of dimension $m$. Using the same notation as before, observe that $\Phi_K^T\colon T\to Q^{m+1}$ is an isomorphism because the vertices $v_0,\dots,v_m$ of $K$ get mapped to vertices $e_0,\dots,e_m$ of $Q^{m+1}$, and the remaining vertex of $T$ gets mapped to the remaining vertex $0$ of $Q^{m+1}$. Also, as before, if we restrict the domain of $\Phi_K^T$ to $K$ and the codomain to $T^m$, we obtain the isomorphism $\Psi_K^K\colon K\to T^m$. So, given $E_K^T$, we may define $E_m$ using the diagram
  \begin{equation*}
    \begin{tikzcd}
      \cF(Q^{m+1})\ar[r,"\left(\Phi_K^T\right)^*","\cong"']&\cF(T)\\
      \oF(T^m)\ar[u, "E_m"]\ar[r, "\left(\Psi_K^K\right)^*", "\cong"']&\oF(K)\ar[u,"E_K^T"]
    \end{tikzcd}    
  \end{equation*}
  In other words,
  \begin{equation*}
    E_m:=\left(\left(\Phi_K^T\right)^*\right)^{-1}\circ E_K^T\circ\left(\Psi_K^K\right)^*\colon\oF(T^m)\to\cF(Q^{m+1}).
  \end{equation*}
  Since $\Phi_K^T\colon T\to Q^{m+1}$ is an isomorphism, it restricts to an isomorphism $\Psi_K^F\colon F\to Q_K^F$ for any $F\le T$, and every subface of $Q^{m+1}$ is of the form $Q_K^F$ for some $F\le T$. Consequently, we can extend the above diagram to
  \begin{equation*}
    \begin{tikzcd}
      \cF(Q_K^F)\ar[r,"\left(\Psi_K^F\right)^*", "\cong"']&\cF(F)\\
      \cF(Q^{m+1})\ar[r,"\left(\Phi_K^T\right)^*","\cong"']\ar[u,"\tr^{Q^{m+1}}_{Q_K^F}"]&\cF(T)\ar[u,"\tr^T_F"]\\
      \oF(T^m)\ar[u, "E_m"]\ar[r, "\left(\Psi_K^K\right)^*", "\cong"']&\oF(K)\ar[u,"E_K^T"]
    \end{tikzcd}    
  \end{equation*}
  The desired properties of $E_m$ in Definition~\ref{def:simpext} then follow from the corresponding properties of $E_K^F$ via these isomorphisms. Specifically, if $Q_K^F=T^m$, then $F=K$, so the composition in the right column is $\tr^T_KE_K^T=E_K^K$, which is the inclusion $\oF(K)\to\cF(K)$. Hence, the composition in the left column, $\tr_{T^m}^{Q^{m+1}}E_m$ is the inclusion $\oF(T^m)\to\cF(T^m)$. Meanwhile, if $T^m\not\le Q_K^F$, then $K\not\le F$, so $\tr_F^TE_K^T=E_K^F=0$, and so $\tr_{Q_K^F}^{Q^{m+1}}E_m=0$ as well.
\end{proof}

\begin{corollary}\label{cor:simpext}
  Let $\cF$ be a simplicial function space that has simplicial extension operators $E_m\colon\oF(T^m)\to\cF(Q^{m+1})$ for every $m<n$ for some nonnegative integer $n$. Then for every simplicial complex $\cT$ of dimension at most $n$, we have a geometric decomposition
  \begin{equation*}
    \cF(\cT)=\bigoplus_{F\in\fT}E_F^\cT\oF(F).
  \end{equation*}
  and, presuming we have chosen spaces $\oF(F)^\dagger$, a dual decomposition
  \begin{equation*}
    \cF(\cT)^*=\bigoplus_{F\in\fT}\left(\tr^\cT_F\right)^*\oF(F)^\dagger.    
  \end{equation*}
\end{corollary}

\begin{proof}
  Combine Theorem \ref{thm:simpext} with Corollaries~\ref{cor:globaldecomp} and \ref{cor:dualdecomp}.
\end{proof}

For the simplicial finite element exterior calculus spaces $\cF=\mathcal P_r\Lambda^k$ and $\cF=\mathcal P_r^-\Lambda^k$, the extension operators constructed in \cite{afw09} for $r\ge1$ are simplicial extension operators. However, for $r=0$, geometric decomposition can fail; we provide this illustrative example.

\begin{example}\label{eg:constant}
  We consider the function space $\cF=\mathcal P_0\Lambda^k$ of $k$-forms with constant coefficients. By dimension, if $m<k$, then $\mathcal P_0\Lambda^k(T^m)=0$, so we have trivial extension maps $E_m$ for $m<k$. However, as we will show, there does not exist a simplicial extension map for $m=k$. First, observe that, by dimension, $\mathcal P_0\Lambda^k(F)=0$ for any proper subface $F$ of $T^k$, so all forms on $T^k$ have vanishing trace. In other words, $\oP_0\Lambda^k(T^k)=\cP_0\Lambda^k(T^k)$. This space is one-dimensional, consisting of constant multiples of the volume form $\omega$ on $T^k$. The definition of $E_k$ would then require that $E_k\omega\in\cP_0\Lambda^k(Q^{k+1})$ have nonzero trace on $T^k$ but zero trace on every other proper subface of $Q^{k+1}$. As we will show, there does not exist a constant $k$-form on $Q^{k+1}$ satisfying these properties.

  Let $\starop\dl_i=(-1)^i\dl_0\wedge\dots\wedge\widehat{\dl_i}\wedge\dots\wedge\dl_k$. These $k$-forms are a basis for $\cP_0\Lambda^k(Q^{k+1})$. So, for an arbitrary $\alpha\in\cP_0\Lambda^k(Q^{k+1})$, we can write $\alpha=\sum_ia_i\starop\dl_i$ for real numbers $a_i$. Let $Q_i$ be the face of $Q^{k+1}$ given by $\lambda_i=0$. Observe that $\tr^{Q^{k+1}}_{Q_i}\starop\dl_j$ is nonzero if $i=j$ and zero if $i\neq j$. Consequently, $\tr^{Q^{k+1}}_{Q_i}\alpha=a_i\tr^{Q^{k+1}}_{Q_i}\starop\dl_i$. If all of these traces are zero, then every $a_i$ is zero, so $\alpha$ is zero, and so its trace on $T^k$ is zero, not $\omega$.

  We conclude that $\cP_0\Lambda^k$ has a geometric decomposition for every triangulation of dimension at most $k$. If the dimension of the triangulation is strictly smaller than $k$, this statement is tautological because $\cP_0\Lambda^k(\cT)=0$. On the other hand, if the dimension of the triangulation is equal to $k$, then $\cP_0\Lambda^k(\cT)$ is equivalent to the space of discontinuous piecewise constant scalar fields. The summands in the geometric decomposition $\cP_0\Lambda^k(\cT)=\bigoplus_{F\in\fT}E_F^\cT\oP_0\Lambda^k(F)$ are nonzero only when $\dim F=k$, so the geometric decomposition simply expresses the observation that $\cP_0\Lambda^k(\cT)$ is a direct sum of the constant fields on each $k$-dimensional element. Finally, if the triangulation has dimension greater than $k$, then local extension operators do not exist.
\end{example}

\section*{Acknowledgments}

This project was supported by NSF award DMS-2411209. I would also like to thank Evan Gawlik and Snorre Christiansen for helpful discussions, as well as the anonymous reviewers.

\printbibliography

\end{document}